\documentclass[11pt]{article}
\usepackage{sigsam, amsmath,amsthm}
\usepackage{ifmtarg}
\usepackage{mathtools}
\usepackage{verbatimbox}
\usepackage{subcaption}

\newtheorem{theorem}{Theorem}

\makeatletter

    \newcommand\contFrac{\@ifstar{\@contFracStar}{\@contFracNoStar}}

    \def\singleContFrac#1#2{%
        \begin{array}{@{}c@{}}%
            \multicolumn{1}{c|}{#1}%
            \\%
            \hline%
            \multicolumn{1}{|c}{#2}%
        \end{array}%
    }
    \def\@contFracStar#1{%
        \mathchoice{
            \@contFracStarDisplay@#1////\@nil%
        }{
            \@contFracStarInline@#1//\@nil%
        }{
            \@contFracStarInline@#1//\@nil%
        }{
            \@contFracStarInline@#1//\@nil%
        }%
    }

    \def\@contFracStarDisplay@#1//#2//#3\@nil{%
        \@ifmtarg{#2}{%
            #1%
        }{%
            #1 + \cfrac{#2}{\@contFracStarDisplay@#3\@nil}%
        }%
    }

        \def\@contFracStarInline@#1//#2\@nil{%
            \@ifmtarg{#2}{%
                #1%
            }{%
                #1 \@@contFracStarInline@@#2\@nil%
            }%
        }
        \def\@@contFracStarInline@@#1//#2//#3\@nil{%
            \@ifmtarg{#3}{%
                + \singleContFrac{#1}{#2}%
            }{%
                + \singleContFrac{#1}{#2} \@@contFracStarInline@@#3\@nil%
            }%
        }
\makeatother

\issue{TBA}
\articlehead{TBA}
\titlehead{Adaptive Thiele interpolation}
\authorhead{Oliver Salazar Celis}
\begin{document}

\title{Adaptive Thiele interpolation}

\author{Oliver Salazar Celis \\
ING Belgium \\
Brussels, Belgium 1000 \\
\url{oliver.salazar.celis@ing.com}}

\date{}

\maketitle

\begin{abstract}
The current implementation of Thiele rational interpolation in Maple (the \verb|ThieleInterpolation| routine) breaks down when the points are not well-ordered. In this article, it is shown how this breakdown can be avoided by ordering the interpolation points in an adaptive way.
\end{abstract}

\section{Introduction}
Maple provides  univariate rational interpolation functionality using the \verb|ThieleInterpolation| routine from the \verb|CurveFitting| toolbox. 
Given $n+1$ distinct complex points $x_0,x_1,\ldots,x_n$ together with finite function valuations $f(x_i)=f_i$ ($i=0,\ldots,n$), this routine produces a continued fraction of the form  
\begin{equation} \label{eqn:thiele}
C_n(x) =  \textstyle \contFrac*{  \varphi_0[x_0]  //  x- x_{0} // \varphi_1[x_0,x_1] // x- x_{1} // \varphi_2[x_0,x_1,x_2] // x- x_{2} // \dots // x- x_{n-1}  // \varphi_n[x_0,\ldots,x_n]},
\end{equation}
where the inverse differences $\varphi_i[x_0,\ldots,x_i]$ ($i=0,\ldots,n$) are obtained from the recursion
\begin{equation} \label{eqn:invdif}
\left\{
\begin{aligned}
 \varphi_0[x_k] &= f_k \qquad &k \geq 0 \\
 \varphi_{i+1}[x_0,\ldots,x_i,x_k] &= \frac{x_k-x_i}{\varphi_{i}[x_0,\ldots,x_{i-1},x_k]-\varphi_{i}[x_0,\ldots,x_i]}\qquad &k > i 
\end{aligned}
\right.  .
\end{equation}
It is well-known~\cite{CUYT1988} that the construction of the inverse differences depends on the order in which the points $x_i$ are taken to constuct~\eqref{eqn:thiele}.   
In some cases, one indeed encounters $\varphi_{i}[x_0,\ldots,x_i] = \infty$ due to division by zero in the recursion~\eqref{eqn:invdif}. This does not mean that the problem itself does not have a solution; a simple reordering of the points can typically resolve this.  The documentation of the \verb|ThieleInterpolation| routine warns the user for this situation, but rather than a reordering suggests a perturbation of the points.  
When division by zero is encountered a message as below is shown 
\begin{verbatim}
Error,(in CurveFitting:-ThieleInterpolation) denominator of zero was produced; 
try perturbing the data points
\end{verbatim}
It is not convenient to leave it up to the user to add perturbations or to reorder the points manually by trial and error. In this article we show that an ordering exists that ensures the existence of the interpolating Thiele continued fraction~\eqref{eqn:thiele}, meaning $\varphi_{i}[x_0,\ldots,x_i] \neq \infty$.

\section{Adaptive selection} \label{greedy} 
The documentation of the \verb|ThieleInterpolation| routine already hints to cases to avoid: \emph{For example, a division-by-zero error is produced when two successive points have the same dependent value or when three successive points are collinear}. We can formalize this observation in the following Theorem.

\begin{theorem}\label{thm:exist}
If the distinct points $(x_i)_{0 \leq i \leq n}$ are ordered such that every two consecutive convergents of the continued fraction~\eqref{eqn:thiele} are different, then $\varphi_{i}[x_0,\ldots,x_i]  \neq \infty$.
\end{theorem}
For a proof, we refer to~\cite[Theorem 3.2]{salazar2021} where it is shown that the inverse differences can essentially be interpreted as the ratio of two (linearized) residuals of successive convergents up to some non-zero factor. 

In light of Theorem~\ref{thm:exist}, the division-by-zero error is avoided if the continued fraction~\eqref{eqn:thiele} can be constructed in such a way that every two consecutive convergents are different. 
One way to achieve this, is to choose the next point $x_{i+1}$ in the construction of $C_{i+1}(x)$ with $0<i < n$ such that the current convergent $C_i(x)$ does not interpolate in $x_{i+1}$, i.e.~$C_{i}(x_{i+1}) \neq f_{i+1}$. Hence, given $C_{i}(x)$ and $(x_0,\ldots,x_i)$, we propose to reorder the remaining points $(x_{i+1}, \ldots,x_n)$ and determine $C_{i+1}(x)$ such that 
$|C_i(x_{i+1}) -f(x_{i+1})|$ is maximal. In this way, $C_n(x)$ is ultimate constructed in an adaptive greedy way by choosing in each step the point where for $0<i<n$ the error between $C_{i}(x)$   and $f(x)$ is maximal. The next theorem shows that this strategy indeed succeeds in the avoidance of infinite inverse differences occurring in~\eqref{eqn:thiele}. 

\begin{theorem}\label{thm:different}
If the distinct points $(x_i)_{0 \leq i \leq n}$ are ordered using a strategy where in each step $C_j(x_{j+1}) \neq f(x_{j+1})$ with $0 \leq j < n$ 
then every two consecutive convergents of the continued fraction~\eqref{eqn:thiele} are different and $\varphi_{i}[x_0,\ldots,x_i]  \neq \infty$.
\end{theorem}
\begin{proof}
The proof is by induction. Let $x_0$ be chosen with $f(x_i) \neq \infty$ and take $x_1$ such that $f(x_0) \neq f(x_1)$. Then from \eqref{eqn:invdif} we have $\varphi_0[x_0] = f(x_0) \neq \infty$ and  
$\varphi_1[x_0,x_1] = (x_1 - x_0)/ \left( f(x_1) - f(x_0)\right) \neq \infty$. Clearly, also $C_0(x) = \varphi_0[x_0] \not \equiv C_1(x) = \varphi_0[x_0] + (x-x_0)/\varphi_1[x_0,x_1]$. 
Assume that the hypothesis holds up to $0< j< n$. We show further below that this necessarily means that $\varphi_{j+1}[x_0,\ldots,x_{j+1}] \neq \infty$, but momentarily we take it for granted. 
 Then assume by contradiction that for some $0< j< n$, 
$$
C_j(x) = \frac{A_j(x)}{B_j(x)} \, \equiv  \, C_{j+1}(x) = \frac{A_{j+1}(x)}{B_{j+1}(x)}
$$
where for $0 \leq i \leq n$ the numerators $A_i(x)$ and denominators $B_i(x)$ satisfy the recurrence relation~\cite{wallis}
\begin{equation} \label{eqn:3term}
\left(
\begin{aligned}
A_i(x) \\
B_i(x)
\end{aligned}
\right)
=\left(
\begin{aligned}
\varphi_{i}[x_0,\ldots,x_i]  A_{i-1}(x) +(x-x_{i-1})A_{i-2}(x)\\
\varphi_{i}[x_0,\ldots,x_i]  B_{i-1}(x) +(x-x_{i-1})B_{i-2}(x)
\end{aligned}
\right),
\end{equation}
with
$$
\left\{
\begin{aligned}
A_{-2}(x) = 0, & \qquad B_{-2}(x) = 1 \\
A_{-1}(x) = 1, & \qquad B_{-1}(x) = 0 \\
A_{0}(x) = \varphi_{0}[x_0] = f(x_0), &\qquad  B_{0}(x) = 1
\end{aligned}
\right. .
$$
This implies that the polynomial $\left[ A_jB_{j+1} -A_{j+1}B_{j}\right](x) \equiv 0$.  However, by construction
$$
 -\left[ f(x_{j+1})B_j(x_{j+1}) -A_j(x_{j+1}) \right] B_{j+1}(x_{j+1}) + \left[ f(x_{j+1})B_{j+1}(x_{j+1}) -A_{j+1}(x_{j+1}) \right] B_{j}(x_{j+1})  \neq 0
$$
because neither $f(x_{j+1})B_j(x_{j+1}) -A_j(x_{j+1}) \neq 0$ nor $B_{j+1}(x_{j+1}) \neq 0$. 
Of these last two inequalities, the first one follows from the selection of $x_{i+1}$ with $C_j(x_{j+1}) \neq f(x_{j+1})$ and the fact that the only common factors of $A_j(x)$ and $B_j(x)$ can be interpolation points used so far (see for instance~\cite{GM80}). For the second inequality, 
 if it was true that $B_{j+1}(x_{j+1}) = 0$, then necessarily also $A_{j+1}(x_{j+1}) = 0$. 
On the other hand~\eqref{eqn:3term} can be written as 
\begin{equation} \label{eqn:matrix}
\begin{pmatrix}
A_{j+1}(x) \\
B_{j+1}(x)
\end{pmatrix}
= 
\prod_{k=0}^{j} 
\begin{pmatrix}
           \varphi_{k}[x_0,\ldots,x_k] 	& x-x_k \\
	   1    & 0
          \end{pmatrix}
          \begin{pmatrix}
\varphi_{j+1}[x_0,\ldots,x_{j+1}] \\
1
\end{pmatrix}.
\end{equation}
Since $x_{j+1} \neq x_k$ for $k=0,\ldots,j$, none of the matrices in~\eqref{eqn:matrix} are singular when putting $x= x_{j+1}$, hence we can write
\begin{equation}\label{eqn:matrixinv}
         \begin{pmatrix}
\varphi_{j+1}[x_0,\ldots,x_{j+1}] \\
1
\end{pmatrix}
= 
\prod_{k=j}^{0} 
\begin{pmatrix}
           0 	& 1 \\
	   1/(x_{j+1}-x_k)    & - \varphi_{k}[x_0,\ldots,x_k]/(x-x_k)
          \end{pmatrix}
 \begin{pmatrix}
A_{j+1}(x_{j+1}) \\
B_{j+1}(x_{j+1})
\end{pmatrix}.
\end{equation}
But if  $B_{j+1}(x_{j+1}) = A_{j+1}(x_{j+1}) = 0 $, then the right-hand side of~\eqref{eqn:matrixinv} cannot equal the left-hand side. Therefore $B_{j+1}(x_{j+1}) \neq  0$. 

What remains to be shown is that $\varphi_{j+1}[x_0,\ldots,x_{j+1}] =  \infty$ cannot occur. Else from \eqref{eqn:invdif} and the induction hypothesis we have $\varphi_{j}[x_0,\ldots,x_{j-1},x_{j+1}]= \varphi_{j}[x_0,\ldots,x_i]\neq \infty$. From application of~\cite[Theorem A.1]{salazar2021} we can then write
$$
\varphi_{j}[x_0,\ldots,x_{i-1},x_{j+1}] = - (x_{j+1}- x_{j-1}) \frac{f(x_{j+1})B_{j-2}(x_{j+1}) -A_{j-2}(x_{j+1}) }{f(x_{j+1})B_{j-1}(x_{j+1}) -A_{j-1}(x_{j+1}) } \neq \infty.
$$
Hence, necessarily  $f(x_{j+1})B_{j-1}(x_{j+1}) -A_{j-1}(x_{j+1}) \neq 0$ so that application of~\cite[Theorem A.1]{salazar2021} is also allowed for $\varphi_{j+1}[x_0,\ldots,x_{j+1}] $ and we have
\begin{equation} \label{eqn:frac}
\varphi_{j+1}[x_0,\ldots,x_{j+1}] = - (x_{j+1}- x_j) \frac{f(x_{j+1})B_{j-1}(x_{j+1}) -A_{j-1}(x_{j+1}) }{f(x_{j+1})B_{j}(x_{j+1}) -A_{j}(x_{j+1}) }. 
\end{equation}
But, similarly as before, due to the selection $C_j(x_{j+1}) \neq f(x_{j+1})$, the denominator of~\eqref{eqn:frac} does not vanish and therefore also $\varphi_{j+1}[x_0,\ldots,x_{j+1}] \neq \infty$.
\end{proof}

{\bf{Remark 1.}} Theorem~\ref{thm:different} does not require the selection of the interpolation points to be greedy per se. 
Nevertheless, from a numerical point of view it remains important to construct the interpolant in as few steps as possible~\cite{CUYT1988}.
A greedy selection is a heuristic way to such an end. The AAA  approach~\cite{Nakatsukasa2018TheAA} for instance also employs such a strategy; the motivation there is not existence but rather for numerical purposes. The remaining freedom is in the choice of the first point $x_0$. One option is to take a point where $|f(x_0)|$ is minimum. As such, at least one zero of $f(x)$ is exactly represented when present in the data.

{\bf{Remark 2.}} Theorem~\ref{thm:different} ensures that $\varphi_{i}[x_0,\ldots,x_i]  \neq \infty$ for those inverse differences appearing in~\eqref{eqn:thiele}. 
This does not exclude the possibility that $\varphi_{i}[x_0,\ldots,x_i] =0$, nor does it prevent intermediate occurrence of $\varphi_{i+1}[x_0,\ldots,x_i,x_k] = \infty$ for $k>i+1$ in~\eqref{eqn:invdif}.
Such intermediate non-finite values are not necessarily a concern to continue the recursion~\eqref{eqn:invdif} using the IEEE 754 standard.
In fact, they are necessary for the occurrence of $\varphi_{i}[x_0,\ldots,x_i] =0$. Such cases are also of no particular concern unless $i=n$, that is, if it occurs for the last inverse difference in $C_n(x)$. This situation can be avoided by adding a stopping criterion when constructing $C_n(x)$. If the maximum absolute error in the remaining points is below a prescribed tolerance, say \verb|tol=5e-15|, then the recursion is stopped. One way to implement this is 
$$
 \max |C_i(x_{k}) -f(x_{k})| < \text{tol} \times \max_{i<j\leq n}(|f(x_j)|) \qquad  \text{for } i < k \leq n.
$$
Essentially it means that, up to the prescribed tolerance,  the underlying function appears rational and there is no further accuracy  gain to be made by adding more interpolation points.

\section{Numerical example}
An important result in rational approximation theory is due to Newman~\cite{newman64}, 
where it is shown that rational approximations can achieve root-exponential convergence $\mathcal{O}(\exp(-C \sqrt{n})$  (with $C>0$) for $f(x) = |x|$ with $x \in [-1,1]$.
This is much faster than what can be achieved with polynomials which converge at an algebraic rate $\mathcal{O}(n^{-1})$ at best.

The aim of this example is to construct Newman approximations, which are rational interpolants to $f(x) = |x|$ in the $2n+1$ points
\begin{equation} \label{newpoints}
(-1, -\eta,\ldots,-\eta^{n-1}, 0, \eta^{n-1},\ldots, 1), \qquad \text{with } \eta = e^{-1/\sqrt{n}}.
\end{equation}
The below Maple code sets up the interpolation data and calls the \verb|ThieleInterpolation| routine 
\begin{verbnobox}[\small]
>with(ArrayTools); with(CurveFitting); 
>N := 5; 
>xleft := j -> -(exp(-1/sqrt(N)))^(j-1);
>xright:= j -> (exp(-1/sqrt(N)))^(N-j);
>xdata := Concatenate(1, Vector(N, xleft), 0, Vector(N, xright));
>ydata := Vector(2*N+1, j-> abs(xdata[j]));
>ThieleInterpolation(xdata, ydata, x)
\end{verbnobox}
Because the \verb|ThieleInterpolation| routine takes the points~\eqref{newpoints} from left to right, this leads to the division-by-zero error  
This is not surprising because the first points (defined by \verb|xleft|) basically lie on the line $y = -x$ which is already reconstructed by $C_1(x) = -x$ from the first two points.
Hence this implementation will fail for any $n>0$.

\begin{figure}[ht!]
     \centering
     \begin{subfigure}[b]{0.44\textwidth}
         \centering
         \includegraphics[width=\textwidth]{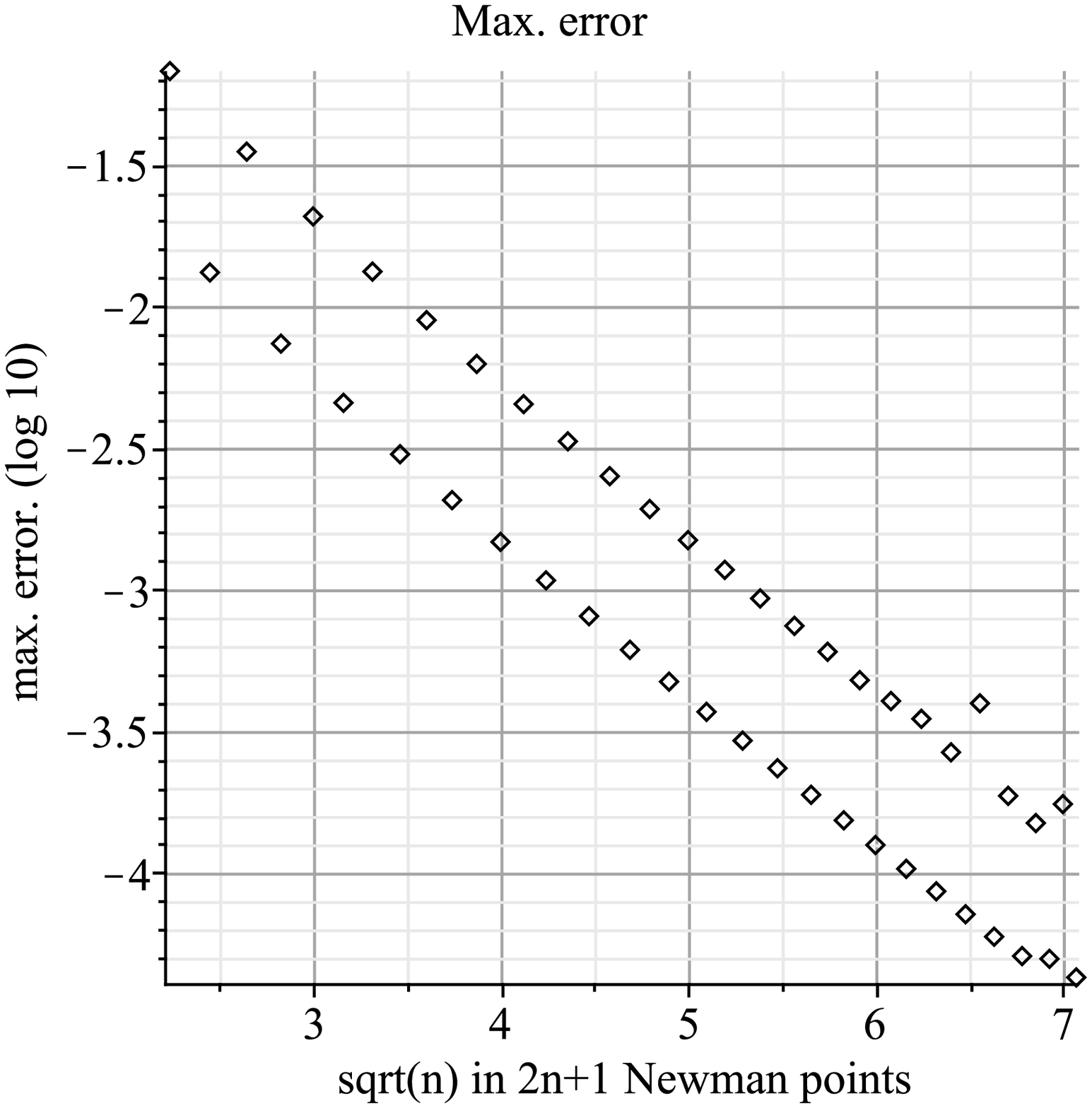}
         \caption{$\|f(x) - C_{2n}(x) \|_\infty$ on $x \in [-1,1] $}
         \label{fig:newmanapprox}
     \end{subfigure}
     \hfill
     \begin{subfigure}[b]{0.44\textwidth}
         \centering
         \includegraphics[width=\textwidth]{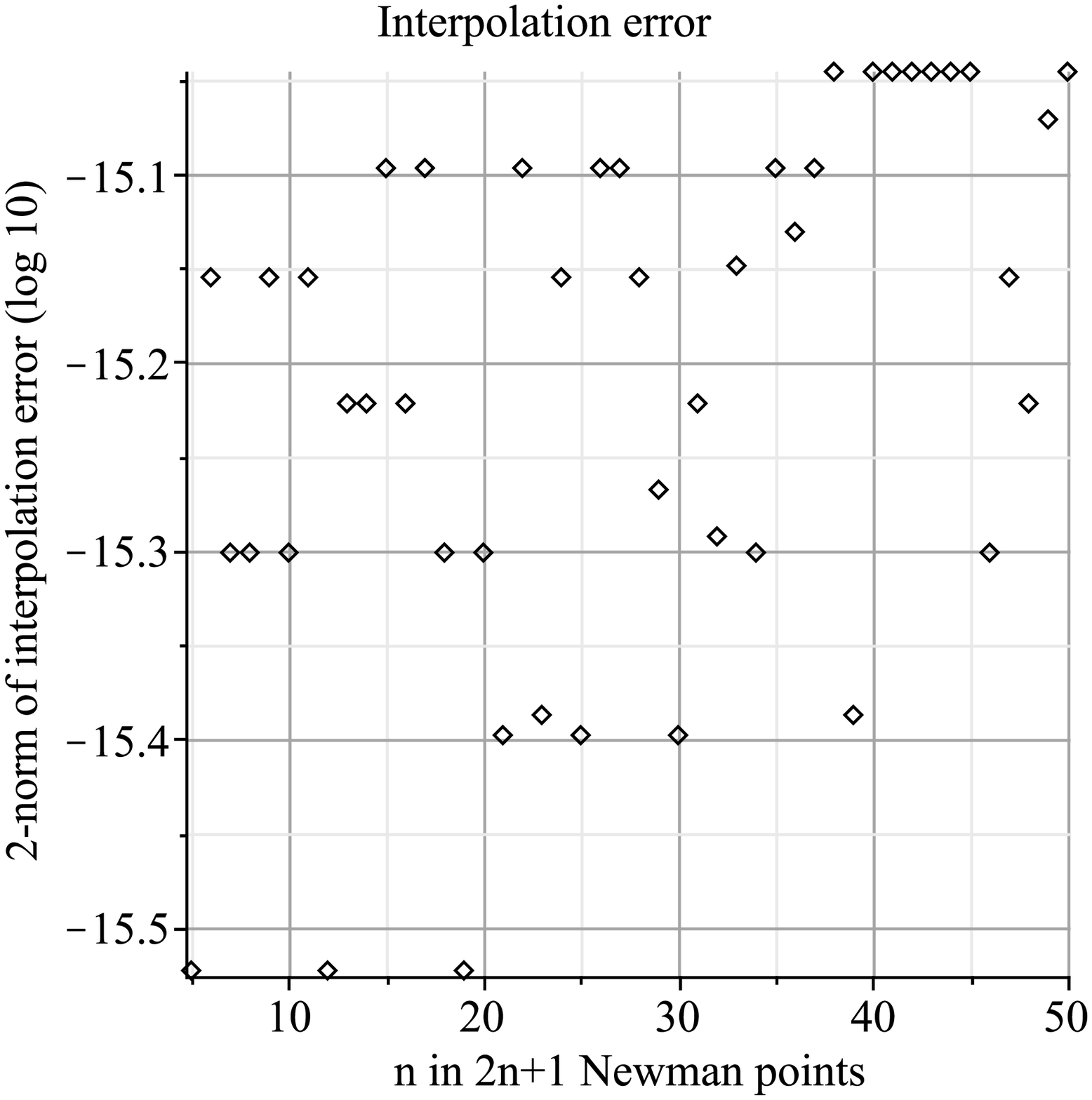}
         \caption{$\| \left( f(x_i)- C_{2n}(x_i) \right)_{0\leq i \leq 2n} \|_2$ }
         \label{fig:forwarderror}
     \end{subfigure}
        \caption{Adaptive Thiele interpolation of $f(x) =|x|$ in Newman points~\eqref{newpoints}  for various $n=5,\ldots,50$.  Left: (discrete) maximum error on $[-1,1]$.
        Right: 2-norm of the interpolation error in the Newman interpolation points~\eqref{newpoints}.}
        \label{fig:absx}
\end{figure}

The situation is completely different when applying the adaptive Thiele interpolation approach. 
A prototype Maple implementation of the greedy strategy is given in Appendix~\ref{sec:maple}. 
Figure~\ref{fig:newmanapprox} shows the maximum error on $[-1,1]$ obtained from Thiele interpolation of $f(x) = |x|$ in Newman points~\eqref{newpoints} for $n=5,\ldots,50$. 
All obtained interpolants use all interpolation points in their construction, meaning that for $n=50$ we have $2n+1 = 101$ interpolation points and we construct $C_{2n}(x) = C_{100}(x)$. 
The root-exponential behavior is clearly visible as a downward sloping straight line when plotting on log10 scale in function of $\sqrt{n}$. 
The maximum error is calculated on a discrete grid of $10000$ points between $0$ and $0.01$ (the maximum error of these interpolants typically occurs near $x=0$ where $f(x) = |x|$ is not differentiable). This discretization gives the impression that convergence is slower for odd $n$ than for even $n$. In fact, odd values of $n$ give approximations with poles in $[-1,1]$, while the even $n$ approximations are pole-free. These observations are in line with those observed in the recent AAA approach~\cite[see Fig.~6.10, p.~1511]{Nakatsukasa2018TheAA}. 

For this run we have put \verb|Digits:=16| to mimic (software) floating-point precision. Recall that by default Maple uses \verb|Digits:=10|. As shown in Figure~\ref{fig:forwarderror} the interpolation error in the Newman points remains highly accurate.

\section{Concluding remarks}
We have shown how the breakdown in the \verb|ThieleInterpolation| routine can be avoided using an adaptive ordering of the interpolation points. 
This actually renders Thiele interpolation into a practical tool for rational interpolation. Particularly when also their poles and zeros can be calculated in a simple fashion as shown in~\cite{salazar2021}.

\appendix
\section{Maple codes} \label{sec:maple}
The ideas above are implemented in Maple with the below prototype code. Mind that the stopping condition is not included here.
\begin{verbnobox}[\small]
restart; with(ArrayTools); with(ListTools); with(LinearAlgebra);

cfrac_eval := proc (aa, zz, xx) 
   description "evaluate continued fraction";
   local N, res, i; 
   N   := Dimension(aa); 
   res := Vector(Dimension(xx), 0); 
   for i from N by -1 to 2 do 
      res := zip(`/`, `~`[`-`](xx, zz[i-1]), `~`[`+`](aa[i], res)) 
   end do; 
   return( `~`[`+`](aa[1], res) )
end proc

cfrac_interpolate := proc (xx, ff) 
   description "adaptive continued fraction interpolation"; 
   local N, rr, k, aa, zz, i, indx_keep, x, f; 
   x := xx; f := ff; N := Dimension(xx); 
   NumericEventHandler(division_by_zero = default); 
   for k to N do 
      if k = 1 then 
         rr := f; i := min[index](abs(rr)); 
         aa := Vector(1, rr[i]); zz := Vector(1, x[i]) 
      else 
         i := max[index](abs(cfrac_eval(aa, zz, x)-f)); 
         rr := zip(`/`, `~`[`-`](x, zz[k-1]), `~`[`-`](rr, aa[k-1])); 
         aa := Append(aa, rr[i]); zz := Append(zz, x[i]) 
       end if; 
       indx_keep := subsop(i = NULL, [seq(1 .. Dimension(x))]); 
       x := x[indx_keep]; 
       f := f[indx_keep]; 
       rr := rr[indx_keep] 
    end do; 
    return Concatenate(2, aa, zz) 
 end proc
\end{verbnobox}

These procedure can then for instance be called on the previous Newman example.
\begin{verbnobox}[\small]
>coefs := cfrac_interpolate(evalf(xdata), evalf(ydata))
>x := evalf(`<,>`(seq((1/1000)*i, i = -1000 .. 1000))); 
>plot(x, cfrac_eval(Column(coefs, 1), Column(coefs, 2), x))
\end{verbnobox}

\bibliographystyle{plain}
\bibliography{cfrac_bib2}
%
%
%

\end{document}